\documentclass[10pt,twoside, reqno]{amsart}
\usepackage[matrix,arrow]{xy}
\usepackage{amssymb}
\usepackage{amsmath}
\usepackage{amsfonts}
\usepackage{epsfig}
\usepackage{color}
\usepackage{epstopdf}
\usepackage{graphicx}
\usepackage{amsthm}
\usepackage{enumerate}
\usepackage[mathscr]{eucal}
\usepackage{verbatim}
\usepackage{bm}

\numberwithin{equation}{section}

\usepackage[bookmarks=true,hyperindex,pdftex,colorlinks,citecolor=blue,
linkcolor=blue,urlcolor=blue]{hyperref}
\usepackage{tikz}
\usetikzlibrary{matrix,arrows.meta}

\newcommand{\vertiii}[1]{{\left\vert\kern-0.25ex\left\vert\kern-0.25ex\left\vert #1 
    \right\vert\kern-0.25ex\right\vert\kern-0.25ex\right\vert}}

\theoremstyle{plain}
\newtheorem{theorem}{Theorem}[section]
\newtheorem{cor}[theorem]{Corollary}

\newtheorem{lemma}[theorem]{Lemma}

\theoremstyle{definition}

\newtheorem{remark}[theorem]{Remark}
\newtheorem{definition}[theorem]{Definition}

\newcommand{\N}{\mathbb{N}}

\DeclareMathOperator{\conv}{conv}
\DeclareMathOperator{\dist}{dist}

\DeclareMathOperator{\re}{Re}

\DeclareMathOperator{\spann}{span}

\DeclareMathOperator{\MA}{\mathcal{MA}}

\begin{document}
\title{Weak minimizing property and reflexivity}

\author[V.~Kadets]{Vladimir Kadets}

\address{ \href{http://orcid.org/0000-0002-5606-2679}{ORCID: \texttt{0000-0002-5606-2679}} {School of Mathematical Sciences, Holon Institute of Technology (Israel)}}
\email{kadetsv@hit.ac.il}

\author[G.~Ribeiro]{Geivison Ribeiro}

\address{ \href{http://orcid.org/0000-0002-0345-7597}{ORCID: \texttt{0000-0002-0345-7597}} { Federal University of Maranhão (UFMA), Brazil}}
\email{geivison.ribeiro@academico.ufpb.br}


\begin{abstract}
For an operator $T:X\to Y$ denote $m(T)=\inf\{\|Tx\|:x\in S_X\}$. A sequence $(x_n)_n$ in $S_X$ is said to be minimizing for $T$ if $\|Tx_n\|$ tends to $m(T)$. In 2020 U.S.~Chakraborty introduced and studied the following \emph{weak minimizing property} (WmP): a pair $(X,Y)$ of Banach spaces is said to have the WmP if, for every bounded linear operator  $T: X \to Y$ that admits a non-weakly null minimizing sequence, the function $x \mapsto \|Tx\|$ attains its minimum on the unit sphere. 
We present the following new results about the WmP for pairs of infinite-dimensional separable Banach spaces:

\begin{enumerate}
\item[(i)]   If $(X,Y)$ has the WmP, then $X$ is reflexive.

\item[(ii)] If $X$ is reflexive and $Y$ does not contain isomorphic copies of $X$, then $(X,Y)$ has the WmP.

\item[(iii)] If $X$ is reflexive and $Y$ contains an isomorphic copy of $X$, then there is an equivalent norm on $Y$ such that, for this equivalent norm,  $(X,Y)$ does not have the WmP.
\end{enumerate}

The first result extends to non-separable $X$ if and only if $X$ possesses a countable total set of functionals.
\end{abstract}
\maketitle

\section{Introduction}

For Banach spaces $X, Y$ over the same field of scalars, denote $L(X,Y)$ the space of all bounded linear operators  $T: X \to Y$ and, for every $T \in L(X,Y)$, denote $\phi_T$ the real-valued function $\phi_T(x) =  \|Tx\|$ defined on the unit sphere $S_X$ of $X$. In this notation, $ \|T\| = \sup\{\phi_T(x): x\in S_X\}$. For the infimum of $\phi_T$ there is no standard notation. In this paper, the notation 
$$
m(T)=\inf\{\|Tx\|:x\in S_X\}
$$ 
from \cite{Cha2020} will be used.

For a finite-dimensional $X$, $\phi_T$ attains its maximum and minimum, but for  infinite-dimensional $X$ this is not always the case. An operator $T \in L(X,Y)$ is called \emph{norm-attaining} if $\phi_T$ attains its maximum. There is a number of results regarding norm-attainment, which happens to be a non-trivial question even for functionals. Say, $X$ is reflexive if and only if every functional $f \in X^*$ attains its norm (R.C.~James \cite{James}); for every $X$ the set of norm-attaining functionals is dense in  $X^*$ (Bishop and Phelps \cite{bis-phe}), see \cite[Chapter 1]{Diestel}. For general operators, the  classification of pairs $(X,Y)$ for which the set of  norm-attaining operators $L_{NA}(X,Y)$ is dense in $L(X,Y)$ was originated by J.~Lindenstrauss \cite{lin2} who in particular presented the first examples in which the density is absent. This study was advanced by many authors and a number of great results were obtained, see the survey \cite{Ac2006} by María D. Acosta. Among these results let us mention the famous J.~Bourgain's theorem \cite{Bour} that $X$ has the Radon-Nikod\'ym property if and only if, for every equivalent norm on $X$ and for every Banach space $Y$, the set $L_{NA}(X,Y)$ is dense in $L(X,Y)$. Despite of extensive study the classification of those pairs that $\overline{L_{NA}(X,Y)} = L(X,Y)$ remains unaccomplished, with a striking example of an open problem of whether $\overline{L_{NA}(X,Y)} = L(X,Y)$ for every two-dimensional $Y$ and every $X$, see discussion in \cite{KLMW}.

An operator $T \in L(X,Y)$ is said to be \emph{min-attaining} if $\phi_T$ attains its minimum. Denote $\MA(X,Y)$ the set of all min-attaining operators $T \in L(X,Y)$. Although both norm attainment and min-attainment are natural optimization problems, the behavior of $\MA(X,Y)$ and $L_{NA}(X,Y)$ is not the same by the simple reason that $\MA(X,Y)$ contains all non-injective operators. We refer the reader to \cite[Section 1.2]{GMMR} by Domingo García, Manuel Maestre, Miguel Martín, and Óscar Roldán for references and a short account of corresponding results.

In this short article the following  \emph{weak minimizing property} is addressed.

\begin{definition}[{U.S.~Chakraborty, \cite[Definition 2.7]{Cha2020}}]
    Let $X$ and $Y$ be Banach spaces.
    \begin{itemize}
        \item For $T\in\mathcal{L}(X,Y)$, a sequence $(x_n)_n$ in $S_X$ is said to {minimize} $T$ if $\|Tx_n\|\to m(T)$. In this case, $(x_n)_n$ is called a {minimizing sequence} for $T$. 
        \item A pair $(X,Y)$ is said to have the {weak minimizing property} (WmP) if every bounded linear operator  $T: X \to Y$ that admits a non-weakly null minimizing sequence is min-attaining.
    \end{itemize}
\end{definition}

The WmP is a direct analog of the \emph{weak maximizing property} (WMP) that was introduced by Richard M. Aron, Domingo García, Daniel Pellegrino and Eduardo V. Teixeira  \cite{AGPT} as follows: a pair of Banach spaces $(E,F)$ is said to have
the WMP if for any $T \in L(E,F)$, the existence of a non-weakly null maximizing sequence for $T$ implies that $T$ attains its norm.

WmP is a useful tool for study of approximation of operators by members of $\MA(X,Y)$. In particular, as it was remarked by Manwook Han \cite{Han2026}, WmP of  $(X,Y)$ implies the density of $\MA(X,Y)$ in $L(X,Y)$.

In the already cited papers \cite{Cha2020} and especially \cite{Han2026} an extensive study of WmP was performed, in particular for all pairs of classical sequence spaces and for many other concrete pairs of spaces it was determined whether they possess the WmP or not. Below we complement this study by exhibiting the connection of  WmP of $(X,Y)$ with the reflexivity of $X$. Namely,
we demonstrate that

\begin{enumerate}
\item[--]   If $(X,Y)$ has the WmP, $X$ possesses a countable total set of functionals and $Y$ is infinite-dimensional, then $X$ is reflexive (Theorem \ref{thm-refl1}).

\item[--] If $X$ is reflexive and $Y$ does not contain isomorphic copies of $X$, then $(X,Y)$ has the WmP (Theorem \ref{thm-refl2}).

\item[--] If $X$ is reflexive and $Y$ contains an isomorphic copy of $X$, then there is an equivalent norm on $Y$ such that, for this equivalent norm,  $(X,Y)$ does not have the WmP  (Theorem \ref{thm-refl2}).
\end{enumerate}
In the preliminary version of the paper (\href{https://arxiv.org/abs/2604.18534v1}{ arXiv:2604.18534v1}), the first result for nonseparable $X$ was stated as an open problem. The solution of that problem by the second author leaded to the strengthening of the result and to the extension of the list of the authors. 

Remark that for $(X,Y)$ with the WMP the reflexivity of $X$ was demonstrated in \cite[Corollary 2.5]{AGPT}, but the proof does not work for WmP because it is based on the existence of not norm-attaining functionals on every nonreflexive space, which is evidently not true for the min-attainment.

As an application in Corollary \ref{thm-square} we present a new solution of \cite[Questions 3.2(2)]{AGPT} that provides a wide range of reflexive spaces $X$ for which $(X,X)$ does not have  the WMP.

\section{The results}

\begin{lemma}\label{lem-refl1}
Let $X$ be a non-reflexive separable Banach space. Then, there is a biorthogonal system $(u_n, u_n^*)_{n \in \N} \subset S_X \times X^*$ such that
\begin{enumerate}
\item[$(1)$] $(u_n)_{n \in \N}$ is a complete system of vectors in $X$ and  $(u_n^*)_{n \in \N} $ is a total system of functionals over $X$ (together this means that  $(u_n)_{n \in \N}$ is an $M$-basis of $X$).

\item[$(2)$]  $(u_n)_{n \in \N}$ is not a weakly null sequence.
\end{enumerate}
\end{lemma}
\begin{proof}
According to D.P. and V.D.~Milman characterization of non-reflexive spaces \cite{MM1964}, there is a basic sequence $(e_n)$ in $X$ and a functional $u^* \in X^*$ such that $\|e_n\| = 1$ and $\inf_n \re u^*(e_n) > 0$ (see \cite[Lemma 1.1]{AcKa2011} for a short proof). Then $(e_n)_{n \in \N}$ is not a weakly null sequence.

It remains to apply V.I.~Gurarii and M.I.~Kadets' Theorem \cite{GuKa1962}, see \cite[Theorem 1.45]{HMVZ2008} which guaranties the existence of an $M$-basis $(u_n)_{n \in \N}$ of $X$ that contains $(e_n)$ as a subsequence.
\end{proof}

\begin{theorem}\label{thm-refl1}
Let $X$ be a non-reflexive Banach space that possesses a countable total set of functionals and $Y$ be an infinite-dimensional Banach space, then  $(X,Y)$ does not have the WmP.
\end{theorem}
\begin{proof} We will consider first the case of separable $X$ (when the existence of a countable total set of functionals is automatic, see \cite[Section 17.2.4, Corollary 2]{kadbook}) and the nonseparable case will come afterwards.  

\vspace{1 mm}
\noindent
\textbf{Proof for the separable $X$}. 
Pick $(u_n, u_n^*)_{n \in \N} \subset S_X \times X^*$ like in Lemma \ref{lem-refl1} and let $(v_n)_{n \in \N}$ be a normalized basic sequence in $Y$ (see \cite[Theorem 1.a.5]{LT1977}). Define $T: X \to Y$ by means of the formula
\begin{equation} \label{eq-def-T}
Tx = \sum_{n=1}^\infty \frac{u_n^*(x)}{2^n\|u_n^*\|} v_n,
\end{equation}
where the absolute convergence of the series and the inequality $\|T\| \le 1$ may be checked as follows:
$$
\|Tx\| \le \sum_{n=1}^\infty \frac{\left|u_n^*(x)\right|}{2^n\|u_n^*\|}\le \sum_{n=1}^\infty \frac{\|u_n^*\|\|x\|}{2^n\|u_n^*\|}=\|x\|.
$$
For this operator, $Tu_n = \frac{ v_n}{2^n\|u_n^*\|}$; consequently
$\lim_{n \to \infty}\|Tu_n\| = 0$. This means that $m(T)=0$ and $(u_n)_{n \in \N}$ is the required minimizing not weakly null sequence. Finally, $T$ is not min-attaining: if $T(x) =0$, then $u_n^*(x)=0$ for all $n \in \N$ which means that $x =0$ (this is exactly the definition of a total system of functionals over $X$).

\vspace{1 mm}
\noindent
\textbf{Proof for the nonseparable $X$}. Let now $(u_n^*)_{n \in \N} \subset X^*$ be a total system of functionals over $X$ with $\|u_n^*\|=1$. The nonseparability of $X$ implies the nonseparability of $X^*$, so we can pick a $g \in X^*$ with $\dist(g, \spann(u_n^*)_{n \in \N}) > 1$. Denote $M_k = \bigcap_{j=1}^k \ker u_j^*$. By the Bipolar Theorem, $M_k^\bot =  \spann\{u_n^*\}_{j=1}^k$. Taking in account the duality between subspaces of $X$ and quotient spaces of $X^*$ \cite[Section 9.4.2, Theorem 1]{kadbook}, we see that 
$$
\left\|g|_{M_k} \right\|_{M_k^*} = \left\|[g] \right\|_{X^*/M_k^\bot } = \dist(g, M_k^\bot) = \dist(g, \spann\{u_n^*\}_{j=1}^k) > 1.
$$
This allows us, for each $n \in \N$, to select a $u_n \in S_{M_n}$ with $g(u_n) > 1$. As in the first part of the proof, pick a normalized basic sequence  $(v_n)_{n \in \N} \subset Y$ and define $T: X \to Y$ by means of the formula \eqref{eq-def-T}. Then
$$
\|Tu_k\| = \left\| \sum_{n=1}^\infty \frac{u_n^*(u_k)}{2^n\|u_n^*\|} v_n\right\| = \left\| \sum_{n=k+1}^\infty \frac{u_n^*(u_k)}{2^n\|u_n^*\|} v_n\right\| \le \sum_{n=k+1}^\infty 2^n \xrightarrow[k \to \infty]{} 0.
$$
Consequently $m(T)=0$ and $(u_n)_{n \in \N}$ is a minimizing sequence for $T$. The condition $g(u_n) > 1$, $n =1,2, \ldots$, implies that $(u_n)_{n \in \N}$ is not weakly null. Finally,  $T$ is not min-attaining by the same reason as in the first part of the proof: if $T(x) =0$, then $u_n^*(x)=0$ for all $n \in \N$; but $(u_n^*)_{n \in \N} \subset X^*$ is total over $X$, which means that $x =0$.
\end{proof}

\begin{remark} \phantom{o}

\begin{enumerate} 
\item  Let $Y = \ell_2$ and $X$ be a non-reflexive space that does not possesses a countable total set of functionals. Then every operator $T\in L(X,Y)$ is not injective so it  is min-attaining, which means that $(X,Y)$ possesses the WmP in spite of the non-reflexivity of $X$. Thus, the condition of existence of a countable total set of functionals on $X$ in Theorem \ref{thm-refl1} cannot be weaken. 

\item Similarly, if $X$ is infinite-dimensional and $\dim Y < \infty$ then every operator $T\in L(X,Y)$ is not injective, so $(X,Y)$ has the WmP independently of reflexivity or non-reflexivity of $X$. This means that the condition  $\dim Y = \infty$ in Theorem \ref{thm-refl1} cannot be dropped neither.

\end{enumerate}
\end{remark}

The next two results deal with the case of reflexive $X$. The first of them is a very simple remark somehow overlooked in the previous papers.

\begin{theorem}\label{thm-refl2}
Let $X$ be reflexive and $Y$ be a Banach space that does not contain an isomorphic copy of $X$, then $(X,Y)$ possesses the WmP.
\end{theorem}
\begin{proof}
Let $T: X \to Y$ be a bounded linear operator that admits a non-weakly null minimizing sequence $(x_n) \subset S_X$, 
\begin{equation} \label{eq_2}
\lim_{n\to\infty} \|Tx_n\| = m(T).
\end{equation}
By our assumption on $Y$, $T$ is not an isomorphic embedding, so $m(T) =0$.
Together with \eqref{eq_2} this says that $\lim_{n\to\infty} Tx_n =0$.

By the reflexivity of $X$, $(x_n)$ contains a subsequence  $(x_{n_k})$ that weakly converges to a non-zero element $x \in X$. Then,
$$
T(x) = w-\lim_{k\to\infty} T x_{n_k} = \lim_{k\to\infty} T x_{n_k} = 0,
$$
which means that $T$ attains its $m(T) =0$ at the point $\frac{x}{ \|x\|} \in S_X$.
\end{proof}

As a simple consequence one gets the WmP for all pairs of the form $(L_p[0,1], \ell_q)$, $1<p<\infty$, $1\leq q <\infty$, except of the pair $(L_2[0,1], \ell_2)$ for which the WmP is also valid but by other argument, see \cite[Theorem 2.9]{Cha2020} where in particular the WmP is demonstrated for pairs of infinite-dimensional separable Hilbert spaces. Remark, that the WmP of the pairs $(\ell_p, \ell_q)$ for $1<p<\infty$, $1 \leq q < \infty$ and $p \neq q$ demonstrated in \cite[Theorem 2.9]{Cha2020} follow from our Theorem \ref{thm-refl2}, but \cite[Theorem 2.9]{Cha2020} includes also the case of $p=q$. The above result covers also several examples from \cite{Han2026}, say $(X_r,Y_0)$  from \cite[Corollary 2.8]{Han2026}, but it does not cover many other positive results from \cite{Han2026} where the specific geometry of the involved spaces plays crucial role. We refer to the classical Lindenstrauss-Tzafriri book \cite{LT1977, LT1979} where an extensive information about embeddability or non-embeddability of classical spaces into one another is given. 

Theorem \ref{thm-refl1} demonstrates a strong isomorphic restriction (reflexivity) on the first space in any pair $(X,Y) \in \mathrm{WmP}$ of infinite-dimensional separable spaces. Now we are going to demonstrate that there are no other possible restriction on $X$: every reflexive $X$ serves as the first space in some pair $(X,Y) \in \mathrm{WmP}$ of infinite-dimensional spaces. 

\begin{cor} \label{thm-restr_X}
Let $X$ be a infinite-dimensional reflexive Banach space and $Y$ be an infinite-dimensional Banach space that does not contain infinite-dimensional reflexive subspaces. Then the pair $(X,Y)$ possesses the WmP. In particular, the pairs $(X,\ell_1)$ and $(X,c_0)$ possess the WmP.
\end{cor}
\begin{proof} The main part evidently follows from Theorem \ref{thm-refl2}. The fact that $\ell_1$ and $c_0$ do not have  infinite-dimensional reflexive subspaces follows from \cite[Proposition 2.a.2]{LT1977}.
\end{proof}

Recall that a separable Banach space is called isomorphically (isometrically) \emph{universal}
if it contains isomorphic (respectively isometric) copies of every separable Banach space. The next corollary says something about possible second space in an WmP pair.

\begin{cor} \label{thm-restr_Y1}
Let $Y$ be a separable infinite-dimensional Banach space that is not isomorphically universal. Then there is an infinite-dimensional reflexive Banach space $X$ such that  $(X,Y)$ possesses the WmP.
\end{cor}
\begin{proof} According to J.~Bourgains's theorem \cite{Bour1980}, see also  \cite[Theorem 2.14]{HMVZ2008},  there is  a  reflexive separable Banach space $X$ which does not embed isomorphically in $Y$. It remains to apply Theorem \ref{thm-refl2} for this pair $(X,Y)$.
\end{proof}

Theorem \ref{thm-refl1} and \ref{thm-refl2} addressed the cases when either $X$ is not reflexive, or $X$ is reflexive and $Y$ does not contain a copy of $X$. The remaining case when $X$ is reflexive and $Y$ contains a copy of $X$ is addressed in the next theorem.

\begin{theorem} \label{thm-restr_Y}
Let $X$ and $Y$ be Banach spaces, $X$ be reflexive and let $Y$ contain an isomorphic copy of $X$. Then there is an equivalent norm on $Y$ such that, for this equivalent norm, the pair $(X,Y)$ does not possess the WmP.
\end{theorem}
\begin{proof}
Let $Z \subset Y$ be the promised subspace isomorphic to $X$ and let $T: X \to Y$, $T(X) = Z$ be the corresponding isomorphic embedding. We are going to build an equivalent norm $\vertiii{\cdot}$ on $Y$ such that, for $\widetilde Y :=(Y, \vertiii{\cdot})$ the operator $\widetilde T: X \to \widetilde Y$ has the following properties that ensure that $(X, \widetilde Y)$ does not possess the WmP:
\begin{enumerate}
\item[$(1)$] $m(\widetilde T) = 1$.

\item[$(2)$] $\widetilde T$ is not min-attaining.

\item[$(3)$] There is a not weakly null sequence  $(u_n)_{n \in \N} \subset S_X$ such that $\lim_{n \to \infty}\vertiii{\widetilde T(u_n)} = 1$.
\end{enumerate}

At first, remark that it is sufficient to define the required norm $\vertiii{\cdot}$ only for elements of $Z$, because every equivalent norm on a subspace extends to an equivalent norm on the whole space (see \cite[Proposition 2.14]{FabHHMZ2011}), and for the properties $(1) - (3)$ only the values of $\vertiii{\cdot}$ on $Z$ are essential.

Let $(e_n)_{n=1}^\infty \subset S_X$ be a basic sequence in $X$. By reflexivity of $X$, $(e_n)$ tends weakly to 0 (weak convergence criterion \cite[17.2.3, Theorem 1]{kadbook} together with the fact that in a reflexive space every basis is shrinking  \cite[Theorem 1.b.5]{LT1977}). Select coefficients $t_n$, $n=2,3, \ldots$ in such a way that $\|\frac12 e_1 + t_n e_n\| = 1 - \frac{1}{2^n}$ and define the unit ball $U \subset Z$ of the required new norm  $\vertiii{\cdot}$ on $Z$ as the closed  convex hull of $T\left(\frac12 B_X \bigcup \left\{\pm\left(\frac12 e_1 + t_n e_n\right)\right\}_{n=2}^\infty\right)$.
Then
$$
\frac12 T(B_X) \subset U \subset  T(B_X) 
$$
which means that $U$ generates an equivalent norm on $Z$ (recall that $T$ is an isomorphism of $X$ and $Z$, so $\frac{1}{\|T^{-1}\|}B_Y \subset T(B_X) \subset \|T\|B_Y$).

Let us demonstrate the properties $(1) - (3)$ of the norm $\vertiii{\cdot}$ on $Z$. At first, the inclusion $ U \subset  T(B_X)$ means that for all $x \in S_X$ we have $\vertiii{Tx} \ge 1$. On the other hand, $T\left(\frac12 e_1 + t_n e_n\right) \in U$, so for the elements 
$$u_n := \frac{\frac12 e_1 + t_n e_n}{\|\frac12 e_1 + t_n e_n\|} \in S_X$$ we have
$$
\vertiii{T(u_n)} =  \frac{1}{\|\frac12 e_1 + t_n e_n\|}\vertiii{T\left(\frac12 e_1 + t_n e_n\right)}= \frac{1}{1 - \frac1n}\vertiii{T\left(\frac12 e_1 + t_n e_n\right)}\le \frac{1}{1 - \frac1n} \xrightarrow[n \to \infty]{} 1.
$$
Together this means that $m(\widetilde T) = 1$ as required in (1). Also, for these $u_n$ the weak limit is equal to $\frac12 e_1$, so $(u_n)$ is not a weakly null sequence and  $\lim_{n \to \infty}\vertiii{\widetilde T(u_n)} = 1$ as required in (3). 

So, it remains to demonstrate (2). Assume to the contrary that there exist a point $v \in S_X$ such that $\vertiii{T(v)}=1$. Then, $T(v) \in U$ which, due to the definition of $U$ means the existence of a sequence 
\begin{align*}
v_n &\in \conv\left(\frac12 B_X \bigcup \left\{\pm \left(\frac12 e_1 + t_n e_n\right)\right\}_{n=2}^\infty\right), 
\\
 v_n &= \lambda_{n,1}\frac12 z + \sum_{k=2}^{m_k} \lambda_{n,k} \left(\frac12 e_1 + t_k e_k\right), z \in B_X, \sum_{k=1}^{m_k} |\lambda_{n,k}| \le 1,
\end{align*}
such that $\vertiii{T(v) - T(v_n)} \xrightarrow[n \to \infty]{} 0$ or, equivalently $\|v - v_n\|  \xrightarrow[n \to \infty]{} 0$. 

Denote $v^*$ a supporting functional of $v$, that is $v^* \in S_{X^*}$ and $v^*(v)=1$. The weak convergence of  $(\frac12 e_1 + t_n e_n)_n$ to $\frac12 e_1$ gives us a number $N\in \N$ such that $|v^*(\frac12 e_1 + t_k e_k)| < \frac34$ for all $k > N$. Consequently,
\begin{align*}
\left|v^*\left(v_n\right)\right| &\le \left| v^*\left(\lambda_{n,1}\frac12 z \right) \right| +\sum_{k=2}^{m_k} \left| v^*\left(\lambda_{n,k} \left(\frac12 e_1 + t_k e_k\right)\right) \right| \\
&\le \frac12|\lambda_{n,1}| + \sum_{k=2}^{N} \left| v^*\left(\lambda_{n,k} \left(\frac12 e_1 + t_k e_k\right)\right) \right| + \sum_{k=N+1}^{m_k} \left| v^*\left(\lambda_{n,k} \left(\frac12 e_1 + t_k e_k\right)\right) \right|  \\
&\le \frac12|\lambda_{n,1}| + \sum_{k=2}^{N} \left|\lambda_{n,k}\right|\left(1 - \frac{1}{2^k}\right) + \sum_{k=N+1}^{m_k} \frac34 \left|\lambda_{n,k}\right| \le \left(1 - \frac{1}{2^N}\right) \sum_{k=1}^{m_k} |\lambda_{n,k}| \le  1 - \frac{1}{2^N}.
\end{align*}
But, on the other hand, $\|v - v_n\|  \xrightarrow[n \to \infty]{} 0$, so $\left|v^*\left(v_n\right)\right|  \xrightarrow[n \to \infty]{} \left|v^*\left(v\right)\right|=1$ which means that we reached a contradiction.
\end{proof}

\begin{cor} \label{thm-restr_Y2}
Let $E$ be a separable infinite-dimensional Banach space that is isometrically universal (for example $E=C[0,1]$). Then for every infinite-dimensional reflexive Banach space $X$ the pair $(X,E)$ does not possess the WmP.
\end{cor}
\begin{proof} According to Theorem \ref{thm-restr_Y}, there is a separable Banach space $Y$ for which  $(X,Y)$ does not have the WmP. Let $T \in L(X,Y)$ be an operator that witnesses this fact. By the universality of $E$ there is an isometric embedding operator $U \in L(Y,E)$. Then  $U \circ T \in L(X,E)$ witnesses the absence of WmP for the pair $(X,E)$. 
\end{proof}

Answering \cite[Questions 3.2(2)]{AGPT}, Dantas, Jung and Mart\'inez-Cervantes
\cite[Corollary 3.1]{DJM} presented a reflexive Banach space $E$ such that $(E,E)$ does not possess the WMP. The following corollary extends the list of such examples.

\begin{cor} \label{thm-square}
Let $E$ be a reflexive infinite-dimensional Banach space that is isomorphic to a square of another Banach space $X$. Then there is an isomorphic copy $\widetilde E$ of $E$ such that  $(\widetilde E,\widetilde E)$ does not possess the WMP.
\end{cor}
\begin{proof} According to Theorem \ref{thm-restr_Y}, there is a Banach space $\widetilde Y$ isomorphic to $X$ for which  $(X,\widetilde Y)$ does not have the WmP. Pick the corresponding isomorphism $\widetilde T \in L(X,\widetilde Y)$ and $(u_n)_{n \in \N} \subset S_X$ that satisfy properties $(1)$--$(3)$ from the proof of Theorem \ref{thm-restr_Y}.  Then  $(\widetilde Y,X)$ does not have the WMP, and $T^{-1} \in L(\widetilde Y,X), \|T^{-1}\|=1$ with $y_n:= \frac{Tu_n}{\|Tu_n\|} \in S_{\widetilde Y}$ witnessing this fact. Let us take $\widetilde E=X \oplus_2 \widetilde Y$ and for every $x + y \in \widetilde E, x \in X, y \in \widetilde Y$, define $U \in L(\widetilde E,\widetilde E)$ as follows: $U(x + y) = T^{-1} y \in X  \subset \widetilde E$. Then $\|U\|=\|T^{-1}\|=1$, $U$ does not attain its norm, and   $(y_n)_{n \in \N} \subset S_{\widetilde E}$ is the required maximizing sequence for $U$ that does not tend weakly to zero. 
\end{proof}

\section*{Acknowledgments}

The author acknowledges support by the KAMEA program administered by the Ministry of Absorption, Israel. The research is partially supported by MICIU/AEI/10.13039/501100011033 and ERDF/EU through the grant PID2021-122126NB-C31


\begin{thebibliography}{99}


\bibitem{Ac2006} María D. Acosta, \emph{Denseness of norm attaining mappings}. RACSAM, Rev. R. Acad. Cienc. Exactas Fís. Nat., Ser. A Mat \textbf{100} (2006), No. 1-2,  9--30. 


\bibitem{AcKa2011} María D. Acosta, Vladimir Kadets,  \emph{A characterization of reflexive spaces}. Math. Ann. \textbf{349} (2011), No. 3,  577--588. 

\bibitem{AGPT} Aron, R.M., Garc\'ia, D., Pellegrino, D., Teixeira, E.V. \emph{Reflexivity and nonweakly null maximizing sequences}, Proc. Amer. Math. Soc. \textbf{148}(2), 741-750 (2020)

\bibitem{bis-phe}
Errett Bishop and R.~R. Phelps, \emph{A proof that every {B}anach space is
  subreflexive}, Bull. Amer. Math. Soc. \textbf{67} (1961), 97--98.

\bibitem{Bour} J.~Bourgain, \emph{On dentability and the Bishop-Phelps property}, Israel J.Math. \textbf{28} (1977), 265--271.

\bibitem{Bour1980}  J.~Bourgain, \emph{On separable Banach spaces, universal for all separable reflexive spaces}, Proc. Amer. Math. Soc. \textbf{79} (1980), 241--246.

\bibitem{Cha2020}  U.~S.~Chakraborty, \emph{Some remarks on minimum norm attaining operators}. J. Math. Anal. Appl. \textbf{492} (2020), no.~2, 124492.

\bibitem{DJM} Dantas, S., Jung, M. and Mart\'inez-Cervantes, G. \emph{Some remarks on the weak maximizing property}, J. Math. Anal. Appl. \textbf{504} (2021), no.~2, 125433. \url{https://doi.org/10.1016/j.jmaa.2021.125433}

\bibitem{Diestel} \textsc{J. Diestel}, \emph{Geometry of Banach
spaces} Lecture notes in Math. \textbf{485}, Springer-Verlag,
    Berlin, 1975.

\bibitem{FabHHMZ2011} M. Fabian, P. Habala, P. H\'ajek, V. Montesinos Santalucía, and V. Zizler, \emph{Banach space theory. The basis for linear and nonlinear analysis}. 
Berlin: Springer (2011)

\bibitem{GMMR} Domingo García, Manuel Maestre, Miguel Martín, and  Óscar Roldán, \emph{On density and Bishop-Phelps-Bollobás-type properties for the minimum norm}, Mediterr. J. Math. \textbf{21} (2024), No. 5, Paper No. 163, 21 pp.

\bibitem{GuKa1962} V.I. Gurarii, M.I. Kadets, \emph{Minimal systems and quasicomplements in Banach spaces}, Sov. Math. Dokl. \textbf{3} (1962), 966--968.

\bibitem{HMVZ2008} P. Hájek, V. Montesinos Santalucía, J. Vanderwerff, V. Zizler, \emph{Biorthogonal systems in Banach spaces}. New York, NY: Springer (2008)


\bibitem{Han2026} Manwook Han, \emph{Weak minimizing property on pairs of classical Banach spaces}.  arXiv:2601.17316

\bibitem{James} James, R. C. \emph{Characterizations of reflexivity}, Stud. Math. \textbf{23} (1964), 205--216.

\bibitem{kadbook} V.~Kadets, \textit{A course in Functional Analysis and Measure Theory}. Translated from the Russian by Andrei Iacob. Universitext. Cham: Springer. xxii, 539~p. (2018).

\bibitem{KLMW}
V.~Kadets, G.~Lopez, M.~Mart\'{\i}n, and D.~Werner,
\emph{Norm attaining operators of finite rank}. In:  Aron, Richard M.; Gallardo Guti\'errez, Eva A.; Martin, Miguel; Ryabogin, Dmitry; Spitkovsky, Ilya M.; Zvavitch, Artem (editors). The mathematical legacy of Victor Lomonosov. Operator theory. Advances in Analysis and Geometry 2. Berlin: De Gruyter, 300~p. (2020), 157--187. 

\bibitem{lin2}
Joram Lindenstrauss, \emph{On operators which attain their norm}, Israel J.
  Math. \textbf{1} (1963), 139--148.


\bibitem{LT1977} Lindenstrauss, J., Tzafriri, L. \emph{Classical Banach spaces. I. Sequence spaces}, Ergebnisse der Mathematik und ihrer Grenzgebiete vol. 92. Springer, Berlin-New York (1977)


\bibitem{LT1979} Lindenstrauss, J., Tzafriri, L. \emph{Classical Banach spaces. II: Function spaces}, Ergebnisse der Mathematik und ihrer Grenzgebiete vol. 97. Springer, Berlin-New York (1979)



\bibitem{MM1964}
D.P. Mil'man, V.D. Mil'man, \emph{Some properties of non-reflexive Banach
spaces}.   Mat. Sb. \textbf{65} (1964), 486--497.



\end{thebibliography}
\end{document}